\documentclass[12 pt]{amsart}

\usepackage{graphicx}
\usepackage{color}
\usepackage{amssymb, amsmath, amsthm}
\usepackage{mathptmx}
\usepackage{comment}

\usepackage{epsfig}
\usepackage{epstopdf}
\usepackage{graphicx}

\theoremstyle{plain}

\newtheorem{theorem}{Theorem}
\newtheorem{lemma}{Lemma}
\newtheorem{corollary}{Corollary}

\newtheorem*{mtheorem}{Main Theorem}

\newtheorem{conjecture}{Conjecture}
\newtheorem{claim}{Case}

\newcommand{\Z}{\mathbb{Z}}

\newcommand{\bi}{\begin{itemize}}
\newcommand{\ei}{\end{itemize}}
\newcommand{\be}{\begin{enumerate}}
\newcommand{\ee}{\end{enumerate}}

\newcommand{\n}{\beta}

\newcommand{\emp}{\emptyset}
\newcommand{\X}{\times}

\newcommand{\eps}{\epsilon}

\newcommand{\A}{\alpha}
\newcommand{\pd}{\partial}

\newcommand{\Pa}{\mathcal{P}}
\newcommand{\cci}{\overline{\chi}}

\numberwithin{definition}{section}
\numberwithin{example}{section}
\numberwithin{lemma}{section}
\numberwithin{theorem}{section}
\numberwithin{corollary}{section}

\begin{document}
\title{Products of Farey graphs are totally geodesic in the pants graph}

\author[S. Taylor]{Samuel J. Taylor}
\address{Department of Mathematics, 
University of Texas at Austin, 
1 University Station C1200, 
Austin, TX 78712, U.S.A.
}
\email{\href{mailto:staylor@math.utexas.edu}{staylor@math.utexas.edu}}

\author[A. Zupan]{Alexander Zupan}
\address{Department of Mathematics, 
University of Texas at Austin, 
1 University Station C1200, 
Austin, TX 78712, U.S.A.
}
\email{\href{mailto:zupan@math.utexas.edu}{zupan@math.utexas.edu}}

\thanks{The first author is partially supported by NSF RTG grants DMS-0636557 and DMS-1148490. The second author is supported by the National Science Foundation under Award No. DMS-1203988.}

\maketitle

\begin{abstract}
We show that for a surface $\Sigma$, the subgraph of the pants graph determined by fixing a collection of curves that cut $\Sigma$ into pairs of pants, once-punctured tori, and four-times-punctured spheres is totally geodesic. The main theorem resolves a special case of a conjecture made in \cite{APS1} and has the implication that an embedded product of Farey graphs in any pants graph is totally geodesic.  In addition, we show that a pants graph contains a convex $n$-flat if and only if it contains an $n$-quasi-flat.
\end{abstract}

\section{Introduction}

The pants graph $\Pa(\Sigma)$ has emerged as an central object in low- \linebreak dimensional geometry and topology over the past decade. The most prominent example of the pants graph's importance is the celebrated result of Brock that the pants graph $\Pa(\Sigma)$ is quasi-isometric to the Teichm\"uller space $\mathcal{T}(\Sigma)$ with its Weil-Petersson (WP) metric \cite{Brock1}.  As a consequence, the large-scale geometry of $\Pa(\Sigma)$ is the same as that of $\mathcal{T}(\Sigma)$, and $\Pa(\Sigma)$ may be used to investigate the geometry of Teichm\"uller space \cite{BMMII, BMar}.  As evidence of the further significance of the pants graph, Brock has shown that distances in $\Pa(\Sigma)$ are coarsely related to volumes of convex cores of quasi-Fuchsian 3-manifolds \cite{Brock1} and volumes of fibered hyperbolic 3-manifolds \cite{Brock2}.  The pants graph has also been used to study the topology of $3$-manifolds. Johnson has developed 3-manifold invariants based on the pants graph of a Heegaard surface \cite{JJheg}, and the second author has produced similar results for knots in 3-manifolds \cite{Zuppy}. \\

While the above results demonstrate some of the striking connections between the pants graph and low-dimensional manifolds, the intrinsic geometry of the pants graph remains relatively unexplored.  Observing that the mapping class group of $\Sigma$ acts naturally on $\Pa(\Sigma)$, we note one important theorem about the rigidity of $\Pa(\Sigma)$, which has been proved by Margalit:

\begin{theorem}\cite{Mar}
Any automorphism of $\Pa(\Sigma)$ is induced by an element of the mapping class group of $\Sigma$.
\end{theorem}

A similar result has been obtained by Aramayona with regard to subsurfaces and inclusion maps.  We define a \emph{multicurve} $Q$ to be a collection of pairwise disjoint simple closed curves in $\Sigma$, and we let $\Pa_Q(\Sigma)$ be the the full subgraph of $\Pa(\Sigma)$ consisting of all pants decompositions containing $Q$.  Aramayona has proved

\begin{theorem}\cite{Ara}\label{include}
Suppose $\Sigma'$ and $\Sigma$ are surfaces, and $i:\Pa(\Sigma') \rightarrow \Pa(\Sigma)$ is an injective simplicial map.  Then there exists a multicurve $Q$ in $\Sigma$ such that $i(\Pa(\Sigma')) = \Pa_Q(\Sigma)$.
\end{theorem}
Thus, in addition to being rigid with respect to automorphisms, the pants graph is rigid with respect to inclusions.  A natural problem which follows is to determine the geometry of the subspaces $\Pa_Q(\Sigma)$ within $\Pa(\Sigma)$.  Although the correspondence between $\Pa(\Sigma)$ and $\mathcal{T}(\Sigma)$ mentioned above is a quasi-isometry, we may look to $\mathcal{T}(\Sigma)$ for clues as to how these subspaces might behave.  By a result of Masur, the boundary of the WP metric completion $\overline{\mathcal{T}}(\Sigma)$ of $\mathcal{T}(\Sigma)$ consists of strata of the form $\mathcal{T}(\Sigma')$, where $\Sigma'$ is the noded Riemann surface obtained by pinching each curve in a multicurve $Q$ \cite{Mas}.  Moreover, each stratum is \emph{totally geodesic} in $\overline{\mathcal{T}}(\Sigma)$: every geodesic in $\overline{\mathcal{T}}(\Sigma)$ with endpoints in $\mathcal{T}(\Sigma')$ is contained entirely in $\mathcal{T}(\Sigma')$ \cite{Wol}. \\

By interpreting these results in terms of $\Pa(\Sigma$), Aramayona, Parlier, and Shackleton have made the following conjecture:

\begin{conjecture}\cite{APS1}\label{aps}
Suppose that $\Sigma'$ and $\Sigma$ are surfaces, and $i:\Pa(\Sigma') \rightarrow \Pa(\Sigma)$ is an injective simplicial map.  Then $i(\Pa(\Sigma'))$ is totally geodesic $\Pa(\Sigma)$.
\end{conjecture}
After applying Theorem \ref{include}, we see that this conjecture is equivalent to the assertion that for every multicurve $Q$ in $\Sigma$, $\Pa_Q(\Sigma)$ is totally geodesic in $\Pa(\Sigma)$.  In order to state which special cases of the conjecture are known, we make several definitions.  The \emph{complexity} $\xi(\Sigma)$ of a compact, orientable, connected surface $\Sigma$ with genus $g$ and $b$ boundary components is $\xi(\Sigma) = 3g + b - 3$, and it is a straightforward exercise to show that $\xi(\Sigma)$ is the number of curves in a pants decomposition of $\Sigma$.  Given a multicurve $Q$ in $\Sigma$, we define the \emph{complementary subsurface} $\Sigma_Q$ of $Q$ in $\Sigma$ to be the components of $\Sigma \setminus \eta(Q)$ which are not pairs of pants, where $\eta(\cdot)$ is an open regular neighborhood. \\

Conjecture \ref{aps} is known to be true in the case that $\xi(\Sigma_Q) = 1$ \cite{APS1}, in the case that $\Sigma_Q$ has two components of complexity one, each with one boundary component interior to $\Sigma$ \cite{APS2}, and in the case that $\Sigma_Q$ is a connected surface of the same genus as $\Sigma$ \cite{ALPS}. \\

We say that a multicurve $Q$ is $(n\X1)$ if $\Sigma_Q$ consists of $n$ components, each having complexity 1.  The main theorem in this paper is the following:

\begin{mtheorem} \label{main}
Let $Q$ be an $(n \X 1)$-multicurve.  Then $\mathcal{P}_Q(\Sigma)$ is totally geodesic in $\mathcal{P}(\Sigma)$.
\end{mtheorem}

Note that if $Q$ is an $(n\X1)$-multicurve, then $\Pa_Q(\Sigma)$ is a graph product of Farey graphs (see Section \ref{subs}).  Thus, applying Theorem C of \cite{Ara}, we obtain the following corollary:

\begin{corollary}
Suppose that $G$ is a graph product of Farey graphs and $\varphi:G \rightarrow \Pa(\Sigma)$ is a simplicial embedding.  Then $\varphi(G)$ is totally geodesic in $\Pa(\Sigma)$.
\end{corollary}

The \emph{rank} $r$ of a metric space $X$ is defined to be the maximal $r$ such that $\Z^r$ quasi-isometrically embeds in $X$.  By \cite{BF}, \cite{BM}, and \cite{Ham}, the rank $r$ of $\Pa(\Sigma)$ is the maximal $n$ such that $\Sigma$ contains an $(n \X 1)$-multicurve, where $r = \lfloor \frac{3g+b-2}{2} \rfloor$.  By fixing an $(r \X 1)$-multicurve $Q$ and a bi-infinite geodesic in each of the $r$ Farey graphs composing the graph product $G = \Pa_Q(\Sigma)$, we have another corollary:

\begin{corollary}\label{flats}
There exists an isometric embedding $\iota: \Z^r \rightarrow \Pa(\Sigma)$ if and only if $r \leq \lfloor \frac{3g + b - 2}{2} \rfloor$.
\end{corollary}

We remark that shortly before the completion of this note, Jos\' e  Est\'evez announced a related result \cite{Est}. He shows the Main Theorem holds if, in addition, one assumes that each boundary component of $\Sigma_Q$ is separating, and as a consequence, he gives an alternate proof of Corollary \ref{flats}.
\\

\noindent \textbf{Acknowledgements} \, The first author thanks Alan Reid for his ongoing support and the second author thanks Cameron Gordon for helpful conversations and insights over the course of this project.

\section{Preliminaries}\label{pre}

Throughout this paper, $\Sigma$ will denote a compact, connected, orientable genus $g$ surface with $b$ boundary components and $3g+b-3 > 0$.  We occasionally use $\Sigma_{g,b}$ when we wish to emphasize $g$ and $b$, and we also call $\Sigma_{g,b}$ a $b$-times-punctured, genus $g$ surface (despite the fact that $\Sigma$ is compact).  We let $\eta(\cdot)$ represent an open regular neighborhood in $\Sigma$.  A \emph{loop} in $\Sigma$ is a simple closed curve and the loop $c$ is \emph{essential} if it is neither \emph{trivial} nor \emph{peripheral}. Recall that $c$ is trivial if it bounds a disk in $\Sigma$ and is peripheral if it is isotopic to a component of $\partial \Sigma$. We use the term \emph{curve} to mean a free isotopy class $[c]$ of an essential loop $c$. For curves $\alpha$ and $\beta$, we denote by $i(\alpha,\beta)$ their \emph{(geometric) intersection number}. This is by definition $\min\{|a\cap b|: a \in \alpha, b \in \beta\}$. The curves $\alpha$ and $\beta$ are \emph{disjoint} if $i(\alpha,\beta) = 0$; otherwise, they \emph{intersect}. \\

A \emph{multicurve} is a collection of pairwise disjoint curves in $\Sigma$. Given any collection of curves in $\Sigma$, we may always choose loop representatives that minimize pairwise geometric intersection numbers by, for example, choosing geodesic representatives in some fixed hyperbolic metric. (This is possible since $3g+b-3 > 0$.) We will make such choices of representatives without further comment. \\

A \emph{pants decomposition} $\nu$ of $\Sigma$ is defined to be a maximal multicurve on $\Sigma$.  Its named is derived from the fact that $\Sigma \setminus \eta(\nu)$ is a collection of  pairs of pants, i.e. copies of $\Sigma_{0,3}$. The complexity $\xi(\Sigma)$ is the cardinality $|\nu|$ of $\nu$, where $\xi(\Sigma) = 3g+b-3$.  An \emph{essential subsurface} $X$ of $\Sigma$ is a compact codimension-$0$ submanifold such that each component of $\partial X$ is nontrivial in $\Sigma$. Note that if $X$ is an essential subsurface of $\Sigma$, then $\xi(X) \le \xi(\Sigma)$ with equality if and only if all boundary components of $X$ are parallel to boundary components of $\Sigma$. For any multicurve $Q$, the \emph{codimension} of $Q$ is defined to be $\xi(\Sigma) - |Q|$.\\

The \emph{pants graph} $\Pa(\Sigma)$ of $\Sigma$ is the graph defined as follows: vertices are pants decompositions, and two pants decompositions $\nu$ and $\nu'$ are connected by an edge whenever they differ by an \emph{elementary move}.  By this we mean that $\nu \cap \nu'$ is a multicurve of codimension one, and letting $\gamma$ and $\gamma'$ denote the unique curves in $\nu \setminus \nu'$ and $\nu' \setminus \nu$, respectively, we have that $\gamma$ and $\gamma'$ intersect in the minimal possible number of times.  Observe that $\Sigma \setminus \eta(\nu \cap \nu')$ is a collection of pairs of pants and a subsurface $Y$ of complexity one (which must contain $\gamma$ and $\gamma'$).  If $Y$ is $\Sigma_{1,1}$, the requirement of minimal intersection number implies that $i(\gamma,\gamma') = 1$; otherwise $Y$ is $\Sigma_{0,4}$ and $i(\gamma,\gamma') = 2$.  See Figure \ref{pantsmove}.  In either case, we call (the isotopy class of) $Y$ the \emph{support} of the pants move from $\gamma$ to $\gamma'$. The pants graph is connected and is equipped with a natural metric $d$ on its vertex set by assigning length one to each edge and defining the distance from $\nu$ to $\nu'$ to be the length of the shortest path connecting $\nu$ to $\nu'$. We note that the above definition holds if $\Sigma$ is the disjoint union of components $\{\Sigma_i\}$.  In this case, $\Pa(\Sigma) = \Pa(\Sigma_1) \X \dots \X \Pa(\Sigma_n)$.\\

\begin{figure}[h!]
  \centering
    \includegraphics[width=.75\textwidth]{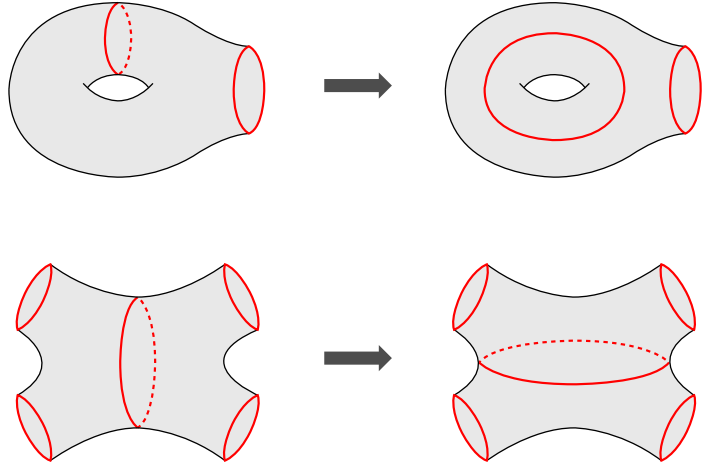}
    \caption{Two possible elementary moves with different supports.}
    \label{pantsmove}
\end{figure}

Suppose that $\nu_0,\nu_1,\dots,\nu_p$ is a path in $\Pa(\Sigma)$, and let $S_i$ denote the support of the $i$th elementary move in the path, $1 \leq  i \leq p-1$.  Define the \emph{support} $X$ of the path $\nu_0,\dots,\nu_p$ to be
\[ X = \bigcup_{i=1}^{p-1} \overline{S_i \setminus \eta(\pd S_i)}.\]
Thus, $\nu_0 \cap X,\dots,\nu_p \cap X$ is a path in $\Pa(X)$, and $\pd X$ is isotopic to curves in $\pd \Sigma \cup (\nu_0 \cap \nu_p)$ (possibly $\nu_0 \cap \nu_p = \emp$ and $\pd X = \pd \Sigma$).  It is important to note that, in general, $X \neq \bigcup S_i$; this is the case in Figure \ref{commute} below.\\

Now, let $\nu_0,\nu_1,\nu_2$ be a path of length two in $\Pa(\Sigma)$ whose support $X$ has two connected components.  Equivalently, $\nu_0,\nu_1$ and $\nu_1,\nu_2$ each differ by an elementary move with supports $S_1$ and $S_2$, respectively, such that the interiors of $S_1$ and $S_2$ are disjoint.  In this case, we say that the elementary moves \emph{commute} and we note that $\nu_0 \cap \nu_2$ is a codimension two multicurve. Define a \emph{commutation of edges} to be the path $\nu_0,\nu_1',\nu_2$, where $\nu_1' = (\nu_0 \cap \nu_2) \cup (\nu_2\setminus \nu_1)\cup (\nu_0 \setminus \nu_1)$ is a pants decomposition since $\nu_2 \setminus \nu_1 \subset S_2$ and $\nu_0 \setminus \nu_1 \subset S_1$. In other words, the commutation of edges is the path obtained by performing the pants move supported in $S_2$ before the pants move supported in $S_1$. See Figure \ref{commute}.  For general edge paths in $\Pa(\Sigma)$, a commutation of edges is a sequence of commutations performed on length two subpaths.  Note that if $\nu_0,\dots,\nu_p$ is a geodesic, then any path $\nu_0,\nu_1',\dots,\nu_{p-1}',\nu_p$ resulting from commutation of edges is also a geodesic with the same support as the original path.  For this reason, we will often abuse notation and suppress the prime notation, renaming the new path $\nu_0,\nu_1,\dots,\nu_{p-1},\nu_p$. \\

\begin{figure}[h!]
  \centering
    \includegraphics[width=.9\textwidth]{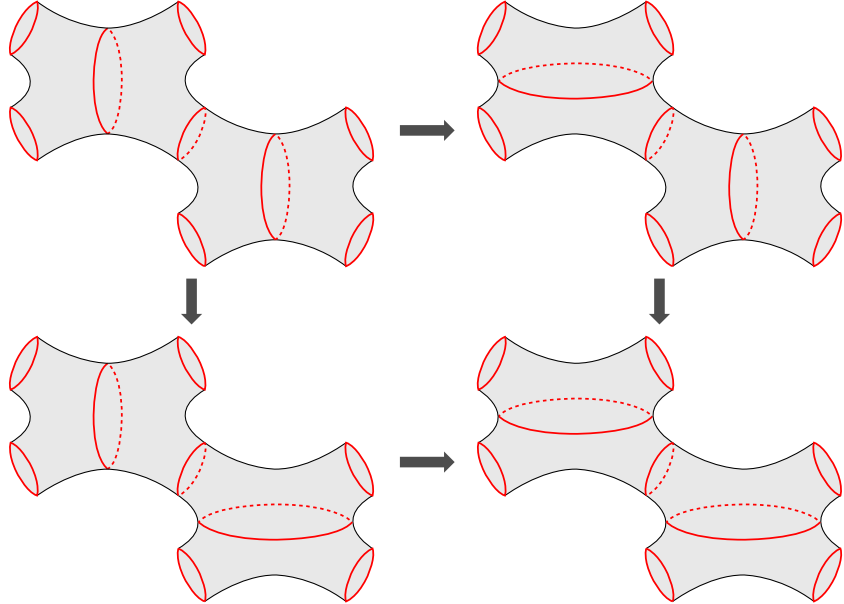}
    \caption{A commutation of edges in $\Pa(\Sigma)$.}
    \label{commute}
\end{figure}

Given a multicurve $Q$ in $\Sigma$, we may associate several different subsurfaces to $Q$.  Recall that the \emph{complementary} subsurface $Y$ of $Q$ is defined as the union of the components of $\Sigma \setminus \eta(Q)$ which are not pairs of pants (if $Q$ is a pants decomposition, then $Y = \emp$). Note that $Q$ uniquely defines $Y$, up to isotopy.  Using this terminology, the complementary subsurface of an $(n \X 1)$-multicurve is by definition a disjoint union of $n$ complexity one subsurfaces.  The support $X$ of a path $\nu_0,\dots,\nu_p$ in $\Pa(\Sigma)$ is the complementary subsurface of $\nu_0 \cap \nu_p$. \\

Let $Q$ be a multicurve in $\Sigma$ with complementary subsurface $Y$.  For any pants decomposition $\nu_Y$ of $Y$, we may associate a pants decomposition $\nu = \nu_Y \cup Q$ of $\Sigma$.  This yields a natural injection $i_Q: \Pa(Y) \rightarrow \Pa(\Sigma)$.  Recall that $\Pa_Q(\Sigma)$ is defined to be the full subgraph of $\Pa(\Sigma)$ spanned by those pants decompositions in $\Pa(\Sigma)$ which contain $Q$, and thus $\Pa_Q(\Sigma) = i_Q(\Pa(Y))$.  For two pants decompositions $\nu,\nu' \subset \Pa_Q(\Sigma)$, we denote their distance in $\Pa_Q(\Sigma)$ as $d_Q(\nu,\nu')$.  The main result in this paper is that $\Pa_Q(\Sigma)$ is totally geodesic in $\Pa(\Sigma)$ when $Q$ is an $(n\X 1)$-multicurve. \\

We will be examining the intersections of curves on $\Sigma$ with subsurfaces of $\Sigma$, and as such we must often deal with properly embedded essential arcs. Thus, we make several more definitions pertaining to arcs.  If $\A$ is a properly embedded essential arc in $\Sigma$ (that is, $\A$ is not isotopic rel $\pd$ into $\pd \Sigma$), we say $\A$ is a \emph{wave} if $\pd \A$ lies in a single component of $\pd \Sigma$ or a \emph{seam} if $\pd \A$ lies in different components of $\pd \Sigma$. When we wish to emphasize that we are interested in the isotopy class of an arc (or a collection of arcs) we use the notation $[ \alpha ]$.

\section{Subsurface projections}\label{subs}

A crucial tool in our proof of the main theorem is the projection of a pants decomposition $\nu$ of $\Sigma$ to a collection of curves in a disjoint union of complexity one subsurfaces.  This projection is a special case of the Masur-Minsky subsurface projections defined in \cite{MMII}. First, suppose that $Y$ is a connected subsurface of $\Sigma$ with $\xi(Y)=1$ and that $\alpha$ is a properly embedded essential arc in $Y$.  Define the projection of $\alpha$ to $Y$, denoted $\pi_Y(\alpha)$, to be the unique curve in $Y$ that misses $\alpha$. If $Y$ is $\Sigma_{1,1}$, then $\A$ is a wave and $\pd \eta(\A \cup \pd Y)$ has two components which are isotopic in $Y$. This curve is $\pi_Y(\A)$.  If $Y$ is $\Sigma_{0,4}$ and $\A$ is a wave, one component of $\pd \eta(\A \cup \pd Y)$ is isotopic into $\pd Y$ and the other component isotopic to $\pi_Y(\A)$.  Otherwise, $\A$ is a seam, $\pd \eta(\A \cup \pd Y)$ has precisely one component, which is essential in $Y$, and this component is $\pi_Y(\alpha)$.   See Figure \ref{project}.  Note that these projections depend only on the free isotopy classes of arcs in $Y$, and for a collection of arcs $Q$ in $Y$, let $[Q] = \{[\A]:\A \in Q\}$. \\

\begin{figure}[h!]
  \centering
    \includegraphics[width=.85\textwidth]{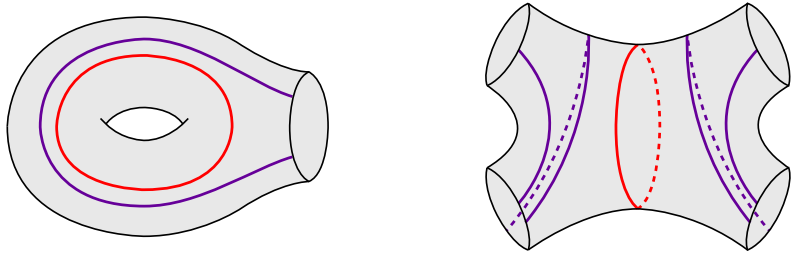}
    \caption{Arcs and their projection curves in $\Sigma_{1,1}$ and $\Sigma_{0,4}$.}
    \label{project}
\end{figure}

For a curve $\beta$ in $\Sigma$, the projection $\pi_Y(\beta)$ is defined as follows: first represent $\partial Y$ and $\beta$ by elements with minimal geometric intersection number (if $\beta$ is isotopic to a component of $\partial Y$ choose its representative to not intersect $Y$).  If $\beta \cap \partial Y = \emptyset$, then either $\beta \subset Y$ and we set $\pi_Y(\beta) =\{\beta\}$, or $\beta \subset \Sigma \setminus Y$ and we set $\pi_Y(\beta) =\emptyset$.  Otherwise, $\beta \cap Y = \{b_1, \ldots, b_k\}$ is a nonempty collection of essential arcs in $Y$ whose isotopy classes depend only on the isotopy class of $\beta$ and set 
$$\pi_Y(\beta) = \bigcup_{1\le i\le k } \pi_Y(b_i).$$
For a multicurve $\nu$, define $\pi_Y(\nu) = \bigcup \pi_Y(v_i)$, where the $v_i$ are the components of $\nu$.  If $\pi_Y(\nu) \neq \emptyset$ then we say $\nu$ \emph{meets} $Y$, and $\nu$ \emph{misses} or \emph{avoids} $Y$ otherwise.\\

Suppose now that $Y$ is the disjoint union of complexity one subsurfaces $Y_1,\dots,Y_n$ of $\Sigma$.  A pants decomposition $\nu$ of $\Sigma$ meets each $Y_i$, and thus we define
\[ \pi_Y(\nu) = \{(w_1,\dots,w_n)\in \Pa(Y): w_i \in \pi_{Y_i}(\nu)\}.\]
As such, we wish to characterize the distance $d_Y(\nu,\nu')$ in $Y$ between projections of $\nu$ and $\nu'$ in $\Pa(\Sigma)$.  However, $\pi_Y(\nu)$ is a set of pants decompositions of $Y$; hence we make the following definition:
\[ d_Y(\nu,\nu') = \max_{\mu \in \pi_Y(\nu)} \left\{ \min_{\mu' \in \pi_Y(\nu')} \left\{d_Y(\mu,\mu') \right\} \right\} \]
where $d_Y(\mu,\mu')$ denotes the distance between $\mu$ and $\mu'$ in $\Pa(Y)$.  If $Q$ and $Q'$ are multicurves or collections of arcs meeting each $Y_i$, we may define $d_Y(Q,Q')$ similarly.  We caution the reader that this distance is, in general, not symmetric.  It does, however, satisfy a triangle inequality: for multicurves $Q$, $Q'$, and $Q''$ meeting each $Y_i$, we have $d_Y(Q,Q'') \leq d_Y(Q,Q') + d_Y(Q',Q'')$.  In addition, we recall that if $\gamma$ is a curve contained in $Y_i$, then $\pi_{Y_i}(\gamma) = \{\gamma\}$; hence, if $Q$ is an $(n\X 1)$-multicurve with complementary subsurface $Y$ and $\nu,\nu' \in \Pa_Q(\Sigma)$, then $\pi_Y(\nu) = \{ \nu \cap \text{int}(Y)\}$ and $d_Y(\nu,\nu') = d_Q(\nu,\nu')$. \\

Note that for a complexity one surface $Y$, the pants graph $\Pa(Y)$ is a Farey graph, shown in Figure \ref{farey}.  If $Y$ is $\Sigma_{1,1}$, an arbitrary multicurve $Q$ in $\Sigma$ intersects $Y$ in at most three isotopy classes of arcs whose projections form a geodesic triangle in $\Pa(Y)$.  Thus, the diameter of $\pi_Y(Q)$ is at most one.  If $Y$ is $\Sigma_{0,4}$, the situation is only slightly more complicated, as described below.

\begin{figure}[h!]
  \centering
    \includegraphics[width=.50\textwidth]{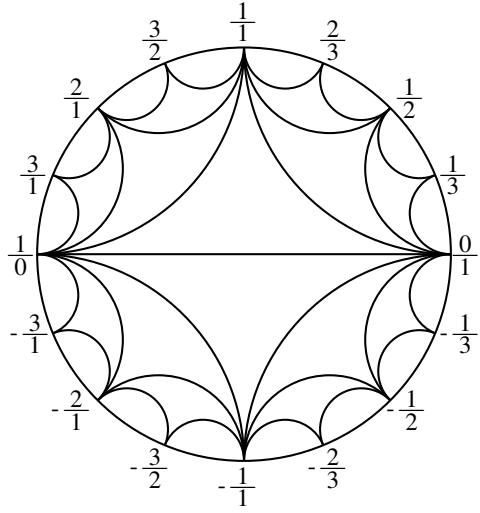}
    \caption{Part of a Farey graph, with edges realized as geodesics in the disk model of $\mathbb{H}^2$.  Two vertices $a/b$ and $c/d$ are connected by an edge whenever $|ad-bc| = 1$.}
    \label{farey}
\end{figure}

\begin{lemma}\label{diam}
Given a multicurve $Q$ in $\Sigma$ and a 4-punctured sphere $Y \subset \Sigma$, the diameter of $\pi_Y(Q)$ is no more than two.  Further, disjoint arcs $\A$ and $\A'$ in $Y$ satisfy $d_Y(\pi_Y(\A),\pi_Y(\A'))=2$ if and only if $\A$ and $\A'$ are nonisotopic seams with boundary in the same two components of $\pd Y$.
\begin{proof}
Let $\A$ and $\A'$ be disjoint nonisotopic arcs in $Y$.  Each component of $\pd Y$ common to $\A$ and $\A'$ contributes at most two points of intersection to $\pi_Y(\A) \cap \pi_Y(\A')$.  Thus, $d(\pi_Y(\A),\pi_Y(\A')) \leq 1$ unless $\A$ and $\A'$ have boundary in two common boundary components.  We may verify that in this case, $|\pi_Y(\A) \cap \pi_Y(\A')| = 4$; thus $d_Y(\pi_Y(\A),\pi_Y(\A')) = 2$ (see Figure \ref{project2}).  If $Q$ is a multicurve in $\Sigma$, then arcs of $Q \cap Y$ are pairwise disjoint, completing the proof.
\begin{figure}[h!]
  \centering
    \includegraphics[width=.4\textwidth]{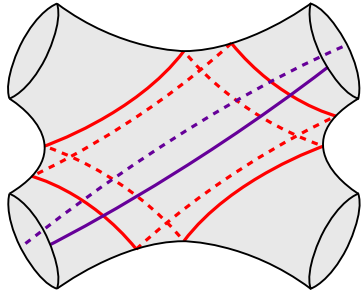}
    \caption{Two disjoint seams in $\Sigma_{0,4}$ that project to curves intersecting four times.}
    \label{project2}
\end{figure}
\end{proof}
\end{lemma}

We require several technical lemmas before proceeding to the proof of the main points.  From here on, we will let $\cci( \cdot)$ denote $-\chi(\cdot)$.  Suppose that $X$ and $Y$ are subsurfaces of $\Sigma$, with $X \cap Y \neq \emp$. We will always assume that these surfaces have been isotoped so that $\partial X$ and $\partial Y$ intersect minimally.  In this context, we call $X \cap Y$ (or one of its components) a \emph{cornered subsurface}.  There is a cell decomposition of $X$ induced by $X \cap \partial Y$ containing $X \cap Y$ as a subcomplex, and we may count the contribution of $X \cap Y$ to $\cci(X)$ in this cell decomposition.  We define
$$\cci_X(Y) =\cci(X \cap Y) - \frac{1}{2}\cci(\text{Fr}_X{Y}),$$
where $\text{Fr}_X{Y}$ (the \emph{frontier} of $Y$ in $X$) denotes $X \cap \partial Y$. \\

Observe that the boundary components of a cornered subsurface are either curves contained in $\pd X$ or $\pd Y$ or $2n$-gons consisting of arcs which alternate between $\text{Fr}_X(Y)$ and $\text{Fr}_Y(X)$.  In addition, a cornered subsurface may not have any ``corners;" that is, all boundary components may curves.  Two particular cornered subsurfaces will be most relevant: If a component $A \subset X \cap Y$ is an annulus such that one component of $\pd A$ is contained in $X$ or $Y$ and the other component is a rectangle, we say $A$ is a \emph{rectangular annulus}.  Similarly, If $P \subset X \cap Y$ is a pair of pants such that  two components of $\pd P$ are contained in $X$ or $Y$ and the other component is a rectangle, we say $P$ is a \emph{rectangular pair of pants}.  Note that $\cci_X(A) = 1$ and $\cci_X(P) = 2$.  See Figure \ref{rectang}.

\begin{figure}[h!]
  \centering
    \includegraphics[width=.6\textwidth]{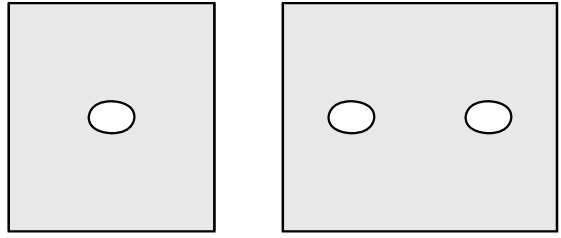}
    \caption{A rectangular annulus and a rectangular pair of pants.}
    \label{rectang}
\end{figure}

\begin{lemma}\label{corner}
For subsurfaces $X$ and $Y$ of $\Sigma$, we have $\cci_X(Y) \leq \cci(X)$.
\begin{proof}
Consider a cell decomposition of $X$ induced by $X \cap \partial Y$ containing $X \cap Y$ as a subcomplex, and let $\{Z_i\}$ denote the closures of the components of $X \setminus \pd Y$.  Thus, each $Z_i$ inherits a cell decomposition,  and either $Z_i \subset X \cap Y$ or $Z_i \subset X \setminus Y$.  Let $f$ denote the number of faces in this decomposition, $v_F$ and $e_F$ the numbers of vertices and edges (respectively) contained in $\text{Fr}_X(Y)$, and $v_Z$ and $e_Z$ the numbers of vertices and edges (respectively) not contained in $\text{Fr}_X(Y)$.  Note that if we compute $\sum \cci(Z_i)$ and $\sum \cci(\text{Fr}_X Z_i)$, vertices and edges contained in $\text{Fr}_X(Y)$ are counted twice, while faces and all other vertices and edges are counted once.  Thus
\begin{eqnarray*}
\sum \cci_X(Z_i) &=& \sum \cci(Z_i) - \frac{1}{2} \sum \cci(\text{Fr}_X Z_i) \\
&=& (-2v_F - v_Z + 2e_F + e_Z - f) -\frac{1}{2}(-2v_F + 2e_F) \\
&=& -v_F - v_Z + e_F + e_Z - f \\
&=& \cci(X).
\end{eqnarray*}
Requiring that $X$ and $Y$ are essential subsurfaces and $|\pd X \cap \pd Y|$ is minimal up to isotopy ensures that $\cci_X(Z_i) \geq 0$ for all $i$.  It follows that
\[ \cci_X(Y) = \sum_{Z_i \subset Y} \cci_X(Z_i) \leq \sum \cci_X(Z_i) \leq \cci(X),\]
as desired.
\end{proof}
\end{lemma}
From the proof of the lemma, we may also conclude that $\cci_X(Y) = \cci(X)$ if and only if $\cci_X(X \setminus \bigcup \eta(Y)) = 0$.  See Figure \ref{corn}. \\

\begin{figure}[h!]
  \centering
    \includegraphics[width=.4\textwidth]{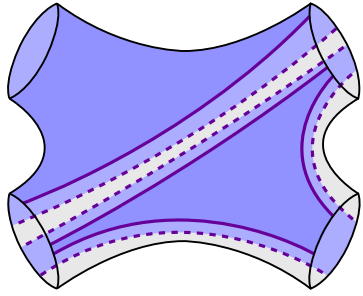}
    \caption{A 4-punctured sphere $X$ containing a rectangular annulus $D_1$ and two hexagons $D_2$ and $D_3$.  Note that $\cci(X) = \cci_X(D_1) + \cci_X(D_2) + \cci_X(D_3)$ in this example.}
    \label{corn}
\end{figure}

In the next two lemmas, we examine the contribution of a complexity one subsurface $Y$ to the Euler characteristic of another subsurface $X$ under some assumptions on the distance from $\pi_Y(c)$ to $\pi_Y(\pd X)$ for a curve $c$ in $X$.  These lemmas will later be used to compare the support $X$ of a path in $\Pa(\Sigma)$ to the complementary subsurface $Y$ of an $(n \X 1)$-multicurve.



\begin{lemma}\label{cont1}
Suppose that $X$ and $Y$ are subsurfaces of $\Sigma$, with $\xi(Y)=1$ and $\pd X \cap Y \neq \emp$.  Let $c$ be a curve contained in $X$ which meets $Y$.  If $d_Y(c,\pd X) = 1$, then there is a cornered component $D$ of $X\cap Y$ such that $\cci_X(D) \geq 1$.
\end{lemma}

\begin{proof}
The key observation here is that if every component $D$ of $X \cap Y$ is a rectangle or a hexagon, then $c$ satisfies $\pi_Y(c) \subset \pi_Y(\pd X)$, as any arc of $c \cap D$ is isotopic in $Y$ to an arc in $\pd X \cap D$.  Hence, if all components of $X \cap Y$ are $2n$-gons, there is a $2n$-gon $D$ with $n \geq 4$; hence $\cci_X(D)  \geq 1$.  See Figure \ref{hexagon}. \\

\begin{figure}[h!]
  \centering
    \includegraphics[width=.7\textwidth]{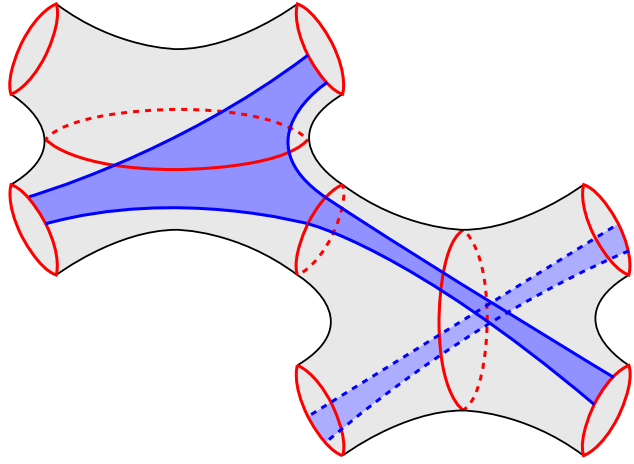}
    \caption{Components $D_1$ and $D_2$ of $X \cap Y$ in $X$, where $D_1$ is a hexagon and $D_2$ is a rectangle.  For an interior curve $c$, any component of $c \cap D_i$ is parallel in $Y$ to a component of $\pd X \cap Y$.}
    \label{hexagon}
\end{figure}

If there is a component $D$ of $X \cap Y$ which is not a topological disk, then $D$ is a punctured sphere or torus, and if $D$ has three or more punctures or is not planar, then $\cci_X(D) \geq \cci(D) \geq 1$.  Thus, suppose that $D$ is an annulus.  If $\pd D$ contains two or more arcs in $\text{Fr}_X(Y)$, then $\cci_X(D) \geq 2 \cdot \frac{1}{2}$.  Otherwise, one component of $\pd D$ is the union of an arc in $\text{Fr}_X(Y)$ and a wave $\A \subset Y$ and the other component is a curve $\gamma$ in $X$.  If $\gamma$ is essential in $Y$, then $\pi_Y(c)=\pi_Y(\pd X) = \{\gamma\}$, a contradiction.  Thus, $\gamma$ is isotopic into $\pd Y$, and it follows that $|\pd Y| >1$, so $Y$ is a 4-punctured sphere.  However, if $c$ intersects $D$, then each component $\n$ of the intersection $c \cap D$ is an essential arc with the property that $\pi_Y(\A) = \pi_Y(\n)$ and $d_Y(c, \pd X) =0$, a contradiction.  See Figure \ref{contrib}.
\begin{figure}[h!]
  \centering
    \includegraphics[width=.4\textwidth]{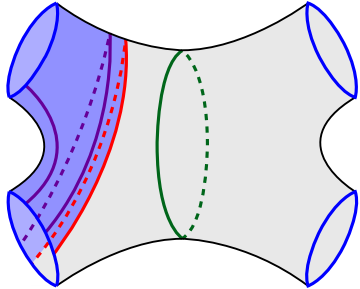}
    \caption{A depiction in $Y$ of the case in which $D$ is an annulus and $\pd D$ contains one arc of $\text{Fr}_X(Y)$.  Note that all arcs pictured project to the same curve in $Y$.}
    \label{contrib}
\end{figure}
\end{proof}

\begin{lemma}\label{cont2}
Suppose that $X$ and $Y$ are subsurfaces of $\Sigma$ with $Y$ a 4-punctured sphere and $\pd X \cap Y \neq \emp$.  Let $c$ be a curve contained in $X$.  If $d_Y(c,\pd X) =2$, then there is a cornered component $D$ of $X \cap Y$ which is a rectangular pair of pants.
\begin{proof}
By Lemma \ref{diam}, some arc $\A$ of $c \cap Y$ is a seam, and every arc of $\pd X \cap Y$ is a seam not isotopic to $\A$ but with endpoints on the same two components $\gamma_1$ and $\gamma_2$ of $\pd Y$ as $\A$.  Any two distinct classes of arcs which are disjoint from $\A$ with endpoints on $\gamma_1$ and $\gamma_2$ must intersect, and thus all arcs of $\pd X \cap Y$ are isotopic in $Y$.  Cutting $Y$ along the arcs of $\pd X \cap Y$ yields some number of rectangles and the desired component $D$.
\end{proof}
\end{lemma}

We present three final technical lemmas before proceeding to the proof of the main theorems.

\begin{lemma}\label{intersections}
Let $P$ be a pair of pants and $Z$ a subsurface of $\Sigma$. Let $D$ be a cornered component of their intersection. If $D$ is an octagon then $D$ intersects all three components of $\partial P$, and if $D$ is a rectangular annulus then $D$ intersects two components of $\partial P$.
\end{lemma}
\begin{proof}
First let $D$ be an octagon. Suppose by way of contradiction that $D$ avoids a boundary component $p$ of $P$.  Then $P \setminus \eta(D)$ has a component $C$ which contains $p$.  It follows that $\cci(P) \geq \cci_P(D) + \cci_P(C) > 1 + 0$, a contradiction to Lemma \ref{corner}. \\

Now let $D$ be a rectangular annulus. First, note that the curve boundary component of $D$ is isotopic into $\pd P$, so $D$ can intersect at most two components of $\pd P$.  Suppose by way of contradiction that $D$ avoids two boundary components of $P$.  Then one of them, call it $p$, is not isotopic to the curve boundary component of $D$, so $P \setminus \eta(D)$ has a component $C$ which contains $p$.  It follows that $\cci(P) \geq \cci_P(D) + \cci_P(C) > 1 + 0$, a contradiction.  
\end{proof}

\begin{lemma}\label{seam1}
Let $X \subset \Sigma$ be a subsurface containing two curves $\gamma_1$ and $\gamma_2$.  Suppose $Y$ is a 4-punctured sphere such that $X \cap Y$ contains a rectangular annulus $A$, where both $\gamma_1$ and $\gamma_2$ meet $A$.  Then
\[ \max \left\{ d_Y(w_1,w_2): w_i \in \pi_Y(\gamma_i \cap A) \right\} \leq 2.\]
\begin{proof}
Up to isotopy, there are up to four possible seams and two possible waves contained in $\gamma_i \cap A$.  See Figure \ref{annrect}.  Let $\A,\n \subset \pd A$ denote the arcs contained in $\pd X$.  By Lemma \ref{diam}, $d_Y(\pi_Y(\A),\pi_Y(\n)) = 2$, and for any arc $\delta_i \subset \gamma_i \cap A$ which is not isotopic to $\A$ or $\n$, we have $d_Y(\pi_Y(\delta_i),\pi_Y(\A)) \leq 1$.  Thus for any two such arcs $\delta_1 \subset \gamma_1 \cap A$ and $\delta_2 \subset \gamma_2 \cap A$,
\[ d_Y(\pi_Y(\delta_1),\pi_Y(\delta_2)) \leq d_Y(\pi_Y(\delta_1),\pi_Y(\A)) + d_Y(\pi_Y(\A),\pi_Y(\delta_2)) \leq 2.\]
\begin{figure}[h!]
  \centering
    \includegraphics[width=.37\textwidth]{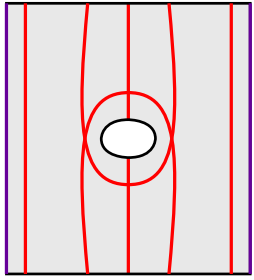}
    \caption{The six possibilities for arcs of $\gamma_i \cap A$ in $A$.}
    \label{annrect}
\end{figure}
\end{proof}
\end{lemma}
\begin{lemma}\label{seam2}
Let $X \subset \Sigma$ be an essential subsurface such that $\nu$ and $\nu'$ are pants decompositions of $X$ related by an elementary move.  Suppose $Y$ is a 4-punctured sphere such that $X \cap Y$ contains a rectangular pair of pants $P$, where $p_1$ and $p_2$ denote the curve boundary components of $P$.  If $\delta$ and $\delta'$ are arc components of $\nu \cap P$ and $\nu' \cap P$ (respectively) that avoid both $p_1$ and $p_2$, then $d_Y(\pi_Y(\delta),\pi_Y(\delta')) \leq 2$.
\begin{proof}
Let $w = \pi_Y(\delta)$ and $w' = \pi_Y(\delta')$, and note that $\pd X \cap Y$ contains a single class of arcs in $Y$, a seam $[\A]$.  If both $\delta$ and $\delta'$ are waves or are isotopic to $\A$, then by Lemma \ref{diam},
\[ d_Y(w,w') \leq d_Y(w,\pi_Y(\A)) + d_Y(\pi_Y(\A),w') \leq 1+1=2.\]
Thus, suppose without loss of generality that $\delta'$ is a seam distinct from $\A$ and consider the possibilities for $\delta$, observing that $\A \cap \delta = \emp$.  If $\delta \cap \delta' = \emp$, the statement holds by Lemma \ref{diam}.  Otherwise, $\delta \cap \delta'$ can be at most two points with signed intersection at most $\pm 1$.  We leave it to the reader to verify that, up to a homeomorphism of $P$, there is precisely one such seam $\delta_1$ and two such waves $\delta_2$ and $\delta_3$ which are candidates for $\delta$.  See Figure \ref{cand}. \\

\begin{figure}[h!]
  \centering
    \includegraphics[width=.32\textwidth]{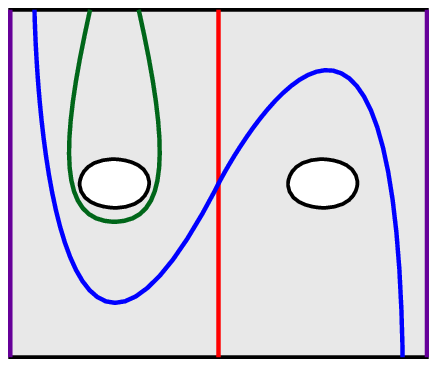}
    \includegraphics[width=.32\textwidth]{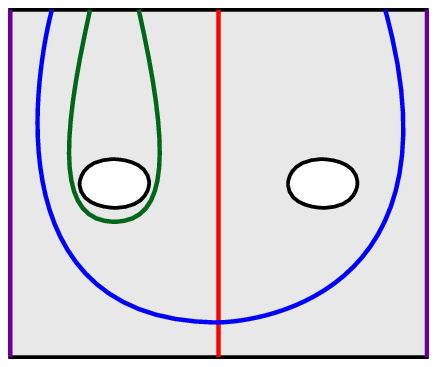}
    \includegraphics[width=.32\textwidth]{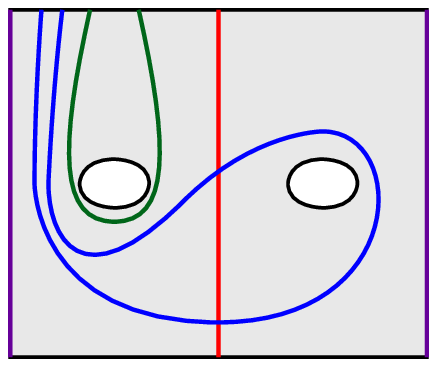}
    \caption{Three possibilities for a seam (left) or wave (middle and right) $\delta$.  Note that in each case there is a wave $\eps$ disjoint from $\delta$ and $\delta'$.}
    \label{cand}
\end{figure}

However, we note that for each $\delta_i$ we may find a wave $\eps$ disjoint from $\delta_i \cup \delta'$, implying by Lemma \ref{diam} that $d_Y(w,w') \leq 2$. This completes the proof.
\end{proof}
\end{lemma}

\section{Proof of the main theorem}
In order to prove the main theorem, we first turn to the proof of Theorem \ref{maintech}.  We will not require the full generality of Theorem \ref{maintech}; however, it may be of independent interest.  It provides a Lipschitz property for the projection of a path in $\Pa(\Sigma)$ to $\Pa(Y)$, where $Y$ is the complementary subsurface of an $(n \X 1)$-multicurve.  The interested reader may compare Theorem \ref{maintech} to Theorem 2 of \cite{APS1} or Theorem 2 of \cite{APS2}.

\begin{theorem}\label{maintech}
Suppose that $Q$ is an $(n \X 1)$-multicurve in $\Sigma$ such that $Y = \Sigma_Q$, and let $\nu_0,\dots,\nu_p$ denote a path in $\Pa(\Sigma)$ such that $p \geq \cci(Y)$.  After a possible commutation of edges, there exists $q$ such that $0 < q \leq p$ and
\[ d_Y(\nu_0,\nu_q) \leq q.\]
\end{theorem}

The proof of Theorem \ref{maintech} is divided into Lemmas \ref{support-no-contain} and \ref{support-contain}.  For both lemmas, we suppose $Y = \cup_{i=1}^n Y_i$, where $\xi(Y_i)=1$.  In addition, let $S_j$ denote the support of the $j$th elementary move and $X_j$ denote the support of the path $\nu_0,\dots,\nu_j$.  Note that $X_0 = \emp \subset X_1 \subset \dots \subset X_p \subset \Sigma$.

\begin{lemma}\label{support-no-contain}
If $X_p$ does not contain a component $Y_i$ of $Y$, then
\[ d_Y(\nu_0,\nu_p) \leq p.\]
\begin{proof}
If $X_p \cap Y_i = \emp$, then $\nu_0 \cap Y_i = \nu_p \cap Y_i$ and $d_{Y_i}(\nu_0,\nu_p) = 0$.  Otherwise $X_p \cap Y_i \neq \emp$ and $Y_i$ contains an arc or curve $\A \subset \pd X_p \subset \nu_0\cap \nu_p$.  Observing that $d_{Y_i}(\nu_0,\nu_p) \leq d_{Y_i}(\nu_0,\pi_{Y_i}(\A))$, we may invoke Lemma \ref{diam}, which implies that $d_{Y_i}(\nu_0,\nu_p) \leq 1$ if $Y_i = \Sigma_{1,1}$ and $d_{Y_i}(\nu_0,\nu_p) \leq 2$ if $Y_i = \Sigma_{0,4}$.  Thus,
\begin{eqnarray*}
d_Y(\nu_0,\nu_p) &=& \sum_{Y_i = \Sigma_{1,1}} d_{Y_i}(\nu_0,\nu_p) + \sum_{Y_i = \Sigma_{0,4}} d_{Y_i}(\nu_0,\nu_p) \\
&\leq&  \sum_{Y_i = \Sigma_{1,1}} 1\,\,+ \sum_{Y_i = \Sigma_{0,4}} 2 \, = \, \cci(Y) \leq p,
\end{eqnarray*}
as desired.
\end{proof}
\end{lemma}

\begin{lemma}\label{support-contain}
If there exists $p' \le p$ such that $X_{p'}$ contains a component $Y_i$ of $Y$, then, after a possible commutation of edges, there exists $q$ such that $0< q \le p'$ and
$$d_Y(\nu_0,\nu_q) \le q. $$

\begin{proof}
Choose an index $m < p'$, such that $X_m$ does not contain a component $Y_i$ of $Y$ but $X_{m+1}$ contains some $Y_i$.  By commuting edges, we may further assume that $X_{m+1}$ is connected.  Note that if $m=0$, then $X_{m+1} = Y_i$ for some $i$ and letting $q = 1$ completes the proof of the theorem.  Thus, we may suppose that $m \geq 1$. \\

Let $\gamma_j$ and $\gamma_j'$ denote the unique curves in $\nu_{j-1} \setminus \nu_j$ and $\nu_j \setminus \nu_{j-1}$, respectively, so that the $j$th elementary move replaces $\gamma_j \in \nu_{j-1}$ with $\gamma_j' \in \nu_j$.  For any $j$, $1 \leq j \leq m$, suppose for the moment that $X = X_j$ is connected, noting that $\xi(X) = |\nu_0 \setminus \nu_j| \leq j$.  In addition, if $g$ is the genus of $X$ and $b = |\pd X|$, then $\xi(X) = 3g + b - 3$, whereas $\cci(X) = 2g + b - 2$.  Thus, $\cci(X) \leq \xi(X) + 1$, with equality if and only if $g = 0$. \\

By assumption, no $Y_i$ is contained in $X$, so for each $Y_i$ that meets $X$, we have $d_{Y_i}(\nu_0,\nu_j) \leq 2$ as in the proof of Lemma \ref{support-no-contain}.  It follows that we may partition $\{Y_i\}$ into
\begin{eqnarray*}
T_0 &=& \{Y_i: d_{Y_i}(\nu_0,\nu_j) = 0\}, \\
T_1 &=& \{Y_i: d_{Y_i}(\nu_0,\nu_j) = 1\}, \\
T_2 &=& \{Y_i: d_{Y_i}(\nu_0,\nu_j) = 2\},
\end{eqnarray*}
and thus $d_Y(\nu_0,\nu_j) = |T_1| + 2|T_2|$.  Let $c_i$ denote any component of $\nu_0$ satisfying $d_{Y_i}(c,\nu_j) = d_{Y_i}(\nu_0,\nu_j)$.  If $Y_i \in T_1,T_2$, then $c_i \subset \text{int}(X)$.  Further, $\partial X \subset \nu_j$ and thus $d_{Y_i}(c_i,\nu_j) \le d_{Y_i}(c_i, \partial X)$.  It follows from Lemmas \ref{cont1} and \ref{cont2} that $\cci_X(Y_i) \geq 1$ if $Y_i \in T_1$ and $\cci_X(Y_i) \geq 2$ if $Y_i \in T_2$.  Therefore $|T_1| + 2|T_2| \leq  \sum \cci_{X}(Y_i)   \le \cci(X)$, and stringing inequalities together yields
\begin{equation}\label{ineq}
 d_Y(\nu_0,\nu_j)= |T_1| + 2|T_2|  \leq  \cci(X) \leq \xi(X) + 1 \leq j+1.
\end{equation}

Of course, it may be the case that for some $j'$ with $1 \leq j' \leq m$, the subsurface $X_{j'}$ of $\Sigma$ is not connected.  However, for each connected component $X$ of $X_{j'}$, we have that $X$ is the union of supports of elementary moves, so we may perform a commutation of edges so that $X = X_j$ for some $j \leq j'$. (Here, $X_j$ is the union of the first $j$ elementary moves occurring in our new path of pants decompositions.)  It follows that for any connected component $X (= X_j)$ of $X_{j'}$, if any of the inequalities in $(1)$ is not sharp, the theorem is proved with $q = j$.  Thus, we may suppose for the remainder of the proof that for such $X$,
{\be
\item $|T_1| + 2|T_2| = \cci(X)$, so that $\cci_{X}(Y_i) = k$ for all $Y_i \in T_k$ ($k=0,1,2$),
\item $\cci(X) = \xi(X) + 1$, so that $X$ is planar, and
\item $\xi(X) = j$, so that $\gamma_k' \neq \gamma_l$ for any $k,l \leq j$.
\ee}

We break the remainder of the proof into a number of possibly overlapping cases, which in total will exhaust all possibilities, proving the statement in question.

{\begin{claim}\label{c1} 
There exists $j \leq m$ such that $X_j \cap \pd Y$ contains a wave.
\begin{proof}
Let $X = X_j$ and suppose that for some $i$, there exists $\gamma \subset \pd Y_i$ such that $\gamma \cap X$ contains a wave $\delta$.  By the above arguments, we may assume $X$ is connected after commuting edges in $\Pa(\Sigma)$.  Since $X$ is planar, $\delta$ separates $X$ into two subsurfaces $X'$ and $X''$ such that $\cci_X(X') \equiv \cci_X(X'') \equiv \frac{1}{2} (\text{mod } 1)$.  However, 
for every $i$, we have by our above assumptions on $X$ that each $Y_i$ contributes a component of integral Euler characteristic to $X$.  It follows that either $\cci_X(X' \setminus \bigcup \{Y_i\})$ or $\cci_X(X'' \setminus \bigcup \{Y_i\})$ is positive; hence $|T_1| + 2|T_2| < \cci(X)$ and the theorem is proved.
 \end{proof}
 \end{claim}}

If a cornered component $D$ of $X_j \cap Y_i$ has a bigon boundary component, then $X_j \cap \pd Y_i$ contains a wave and the theorem holds by Case \ref{c1}.  Thus, we may assume from this point forward that no $D$ has a bigon boundary component.  For any $j$, with $1 \leq j \leq m$, suppose that $X_j \cap Y_i$ contains a pair of pants component $P$ such that $\pd P \subset \pd X_j \cup \pd Y_i$.  We call such $P$ a \emph{full pair of pants}.

{\begin{claim}\label{c2}
There is a $j$ with $1 \leq j \leq m$ such that $X_j \cap Y$ contains a full pair of pants.
\begin{proof}
Suppose $X_j$ is connected and $X_j \cap Y_i$ contains a full pair of pants $P$.  Then a component $\gamma$ of $\pd X_j$ is essential in $Y_i$, implying that $\pi_{Y_i}(\nu_0) = \pi_{Y_i}(\nu_j) = \{\gamma\}$.  Hence $d_{Y_i}(\nu_0,\nu_j) = 0$ but $\cci_{X_j}(Y_i) \geq 1$, and again $|T_1| + 2|T_2| < \cci(X_j)$.
\end{proof}
\end{claim}}

{\begin{claim}\label{c3}
The curve $\gamma_{m+1} = \nu_m \setminus \nu_{m+1}$ is separating in $X_{m+1}$.
\begin{proof}
Suppose without loss of generality that $Y_1 \subset X_{m+1}$.  If $X_m \cap \pd Y_1 = \emp$, then $X_m \cap Y_1 = \emp$, as the only nontrivial subsurfaces of $Y_1$ have at least one boundary component isotopic into $\pd Y_1$.  Thus, $X_{m+1}$ is the disjoint uniont of $X_m$ and $S_{m+1}$ and $Y_1 \subset S_{m+1}$.  In this case we may commute edges in $\Pa(\Sigma)$ so that $Y_1 \subset S_1$ and $m=0$, which is addressed above.  Thus, suppose $X_m \cap \pd Y_1 \neq \emp$.  If $\pd Y_1 \subset X_m$, then either $Y_1$ is contained in a component of $X_m$ (which we have assumed does not occur) or $X_m \cap Y_1$ contains a full pair of pants $P$ (which refers us to Case \ref{c2}).  Hence, assume that $\pd X_m \cap \pd Y_1 \neq \emp$.  Since $\pd Y_1 \subset X_{m+1}$, we have that if $\gamma_{m+1}$ separates $X_{m+1}$, then there is a component $X'$ of $X_m$ such that $\gamma_{m+1}$ is isotopic in $\Sigma$ to a single component $\gamma$ of $\pd X'$, and $\pd X' \cap \pd Y_1 \subset \gamma \cap \pd Y_1$, implying that $X' \cap \pd Y_1$ contains a wave as in Case \ref{c1}.  
\end{proof}
\end{claim}}

By ruling out Case \ref{c3}, we may assume from this point forward that $\gamma_{m+1}$ is nonseparating in $X_{m+1}$, and, as a consequence, $X_m$ is connected and $\gamma_{m+1}$ is isotopic in $\Sigma$ to two distinct components of $\pd X_m$.  In addition, $\cci(X_m) = \cci(X_{m+1})$, $\gamma_k' \neq \gamma_l$ for any $k,l \leq m$, and by (\ref{ineq}) above, we have $d_Y(\nu_0,\nu_m) \leq m+1$. Recall our assumption that $Y_1$ is contained in $X_{m+1}$. \\

{\begin{claim}\label{c4}
$Y_1 \subset X_{m+1}$ is a punctured torus.
\begin{proof}
Since $\cci_{X_m}(Y_1) = \cci_{X_{m+1}}(Y_1) = 1$, our previous assumptions imply $d_{Y_1}(\nu_0,\nu_m) = 1$ and there is a component $R$ of $X_m \cap Y_1$ such that $\cci_{X_m}(R) = 1$.  If $R$ is a full pair of pants, we refer to Case \ref{c2}.  Since $|\pd Y_1|=1$, $R$ is not a rectangular annulus; hence suppose that $R$ is an octagon.  If $X_m \cap \pd Y_1$ contains a wave, refer to Case \ref{c1}.  Otherwise $R \cap \pd Y_1$ consists of four seams in $X_m$ connecting the two components $\gamma'$ and $\gamma''$ of $\pd X_m$ that are isotopic to $\gamma_{m+1}$.  Any two of these seams separate the planar surface $X_m$; hence four such seams cannot cobound an octagon $R$.
\end{proof}
\end{claim}}

{\begin{claim}\label{c5}
$Y_1 \subset X_{m+1}$ is a 4-punctured sphere.
\begin{proof}
The bulk of the proof of the theorem will be devoted to this case, which can easily be seen to exhaust all possibilities.  Since $\cci_{X_m}(Y_1) = \cci_{X_{m+1}}(Y_1) = 2$, our previous assumptions imply $d_{Y_1}(\nu_0,\nu_m) = 2$ and there is a component $R$ of $X_m \cap Y_1$ such that $\cci_{X_m}(R) = 2$.  Lemma \ref{cont2} asserts that $R$ is a rectangular pair of pants such that $\pd X_m \cap Y_1$ contains a single class of arc, a seam $[\A]$.  Let $q_1$ and $q_2$ denote the curve components of $\pd R$ contained entirely in $X_m$, noting that $q_1$ and $q_2$ are isotopic to curves in $\pd Y_1$. \\ 

By arguments in the proof of Lemma \ref{support-no-contain}, we have $d_{Y_i}(\nu_0,\nu_{m+1}) \leq \cci_{X_{m+1}}(Y_i)$ if $Y_i$ is not contained in $X_{m+1}$.  If for every $i$, $d_{Y_i}(\nu_0,\nu_{m+1}) \leq \cci_{X_{m+1}}(Y_i)$ then as in (\ref{ineq}) above,
\begin{equation}\label{m-plus-one}
d_Y(\nu_0,\nu_{m+1}) \leq \cci(X_{m+1}) = \cci(X_m) \leq \xi(X_m) + 1 \leq m+1,
\end{equation}
completing the proof.  Otherwise, we may suppose without loss of generality that for, say, $Y_1$ we have $d_{Y_1}(\nu_0,\nu_{m+1}) > \cci_{X_{m+1}}(Y_1)$.  It follows that $Y_1$ must be contained in $X_{m+1}$, and we choose $Y_1$ so that $d_{Y_1}(\nu_0,\nu_{m+1})$ is maximal over $\{Y_i\}$.  Note that $\gamma_{m+1} \in \nu_m$ intersects $R$ in two arc boundary components isotopic to $\A$ in $Y_1$. 
By assumption, every curve of $\nu_m \cap \text{int}(R)$ is one of the curves $\gamma'_k$, where $k \leq m$, and as $\gamma'_k \neq \gamma_{m+1}$, we have $\gamma'_k \in \nu_{m+1}$.  If some $\gamma_k'$ meets $R$ in an arc or curve $\delta$, then $\delta$ is also a component of $\nu_{m+1} \cap Y_1$.  We will show that either one of the previous cases holds, or there is such a $\gamma_k'$ and $\delta \subset \gamma_k' \cap R$ satisfying $d_{Y_1}(\nu_0,\delta) \leq 2$.  In this case, $d_{Y_1}(\nu_0,\nu_{m+1}) \leq 2$ and (\ref{m-plus-one}) holds, completing the proof. \\



Suppose that for some $j < m$, $q_1 \cap X_j$ contains an arc.  Since $q_1 \subset X_m$, we may choose $j$ such that $q_1 \cap X_j$ contains an arc but $q_1$ is contained entirely within $X_{j+1}$.  As $X_{j+1}$ is planar, $\gamma_j$ is separating in $X_{j+1}$, and there is a component $X'$ of $X_j$ such that $\gamma_j$ is isotopic to a component of $\pd X'$ containing $q_1 \cap \pd X'$.  Hence, $q_1 \cap X'$ contains a wave, and we refer to Case \ref{c1}.  A parallel argument shows the same to be true if $q_2 \cap X_j$ contains an arc. \\


We now undertake a careful analysis of the stages at which $d_{Y_1}(\nu_0,\nu_j)$ grows with $j$.  To this end, let $r_1$ be the smallest index for which $d_{Y_1}(\nu_0,\nu_{r_1}) > 0$.  This implies $\cci_{X_{r_1-1}}(Y_1) = 0$ and $\pi_{Y_1}(\nu_0) \subset \pi_{Y_1}(\nu_{r_1-1})$.  After commutations, we may assume that $X_{r_1-1}$ is connected.  Suppose first that $d_{Y_1}(\nu_0,\nu_{r_1}) = 2$.  Then $\cci_{X_{r_1}}(Y_1) = 2$; hence $\cci(X_{r_1} \setminus X_{r_1-1}) \geq 2$, which implies that $X_{r_1}$ is the disjoint union of $X_{r_1-1}$ and $S_{r_1}$.  Thus, $S_{r_1} \cap Y_1$ is a rectangular pair of pants, $q_1$ and $q_2$ are isotopic into $\pd S_{r_1}$, and $\gamma_{r_1}'$ meets $R$ in an arc or curve $\delta$.  Applying Lemma \ref{seam2} with $X = X_{r_1}$, for any component $\delta'$ of $\nu_{r_1-1} \cap Y_1$, we have $d_{Y_1}(\delta',\delta) \leq 2$.  Thus, $d_{Y_1}(\nu_0, \delta) \leq d_{Y_1}(\nu_{r_1-1},\delta) \leq 2$, completing the proof. \\

On the other hand, suppose that $d_{Y_1}(\nu_0,\nu_{r_1}) =1$.  As $X_{r_1-1}$ is connected, $\overline{X_{r_1} \setminus X_{r_1-1}}$ is either a pair of pants (if $X_{r_1-1}$ and $\text{int}(S_{r_1})$ overlap) or a 4-punctured sphere (if $X_{r_1-1}$ and $S_{r_1}$ are disjoint).  By Lemma \ref{cont1} and our previous assumptions, there is a component $Q_1$ of $X_{r_1} \cap Y_1$ such that $\cci_{X_{r_1}}(Q_1)=1$.  If $Q_1$ is a full pair of pants, we refer to Case \ref{c2}, so we may assume that $Q_1$ is either an octagon or a rectangular annulus.  In addition, $d_{Y_1}(\nu_{r_1 - 1},\nu_{r_1}) > 0$ implies that there is an arc $\mu \subset \gamma_{r_1} \cap Y_1$ such that $\pi_{Y_1}(\mu) \notin \pi_{Y_1}(\nu_{r_1})$.  In particular, this implies that $Q_1$ meets $\gamma_{r_1}$.  If $\overline{X_{r_1} \setminus X_{r_1-1}}$ is a pair of pants, then $Q_1 \cap X_{r_1-1}$ is a collection of rectangles joining \emph{distinct} boundary components of $X_{r_1-1}$, or else $X_{r-1} \cap \pd Y_1$ contains a wave and we refer to Case \ref{c1}.  In the former case, each arc of $\gamma_{r_1} \cap Y_1$ is parallel in $Y_1$ to an arc of $(\nu_0 \cap \nu_{r_1}) \cap Y_1$, and thus $\pi_{Y_1}(\nu_0) \subset \pi_{Y_1}(\nu_{r_1})$, a contradiction (see Figure \ref{distint}).  It follows that $\overline{X_{r_1} \setminus X_{r_1-1}}$ is the 4-punctured sphere $S_{r_1}$. \\

\begin{figure}[h!]
  \centering
    \includegraphics[width=.6\textwidth]{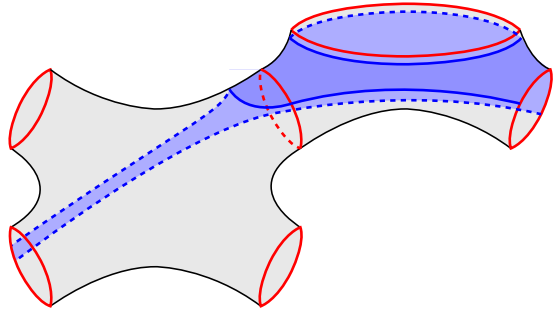}
    \caption{The case in which $\overline{X_{r_1} \setminus X_{r_1 - 1}}$ is a pair of pants.  Here $\gamma_{r_1} \cap Y_1$ is parallel to an arc of $\pd X_{r_1} \cap Y_1$.}
    \label{distint}
\end{figure}

If a boundary arc of $Q_1 \subset R$ meets $q_1$ or $q_2$, then $q_1 \cap X_{r_1}$ or $q_2 \cap X_{r_1}$ contains an arc, completing the proof of the theorem as described above.  Otherwise, boundary arcs of $Q_1$ avoid $q_1$ and $q_2$.  Suppose that $\gamma_{r_1}'$ meets $Q_1$ in an arc $\delta$ and let $\delta'$ be any arc of $\nu_{r_1-1} \cap Y_1$.  If $\delta' \cap \delta = \emp$, then $d_{Y_1}(\delta',\delta) \leq 2$ by Lemma \ref{diam}.  On the other hand, if $\delta' \subset Q_1$ and $Q_1$ is a rectangular annulus, Lemma \ref{seam1} asserts that $d_{Y_1}(\delta',\delta) \leq 2$.  If $\delta' \subset Q_1$ and $Q_1$ is an octagon, then we note that $\nu_{r_1-1} \cap X_m$ and $\nu_{r_1} \cap X_m$ are pants decompositions of $X_m$, $\delta' \subset \nu_{r_1-1} \cap R$, $\delta \subset \nu_{r_1} \cap R$, and both $\delta$ and $\delta'$ avoid $q_1$ and $q_2$.  In this case we apply Lemma \ref{seam2}, which asserts $d_{Y_1}(\delta',\delta) \leq 2$.  In any case, $d_{Y_1}(\nu_0,\delta) \leq d_{Y_1}(\nu_{r_1-1},\delta) \leq 2$, completing the proof. \\

Thus, we may assume that $\gamma_{r_1}'$ avoids $Q_1$, and since $Q_1 \subset S_{r_1}$ and $\gamma_{r_1}'$ cuts $S_{r_1}$ into two pairs of pants, $Q_1$ cannot be an octagon by Lemma \ref{intersections}.  For the remainder of the proof, we suppose that $Q_1$ is a rectangular annulus and let $q$ denote the boundary curve of $Q_1$.   As $\gamma_{r_1}'$ avoids $Q_1$, it must be isotopic to $q$ by Lemma \ref{intersections}.  In addition, $q \subset \text{int}(X_{r_1})$ implies $q\subset \pd Y_1$ and thus $q \subset \pd R$.  See Figure \ref{first}.  Suppose without loss of generality that $q = q_1$.  Since $Q_1$ is a rectangular annulus, there is an arc component, call it $\n$, of $\pd Q_1 \subset \nu_{r_1} \cap Y_1$ which is not isotopic to $\A$.  Note that $\beta$ is a seam in $Y_1$ that separates $q_1$ from $q_2$ in $R$.  See Figure \ref{RRR}. \\

\begin{figure}
  \centering
    \includegraphics[width=1 \textwidth]{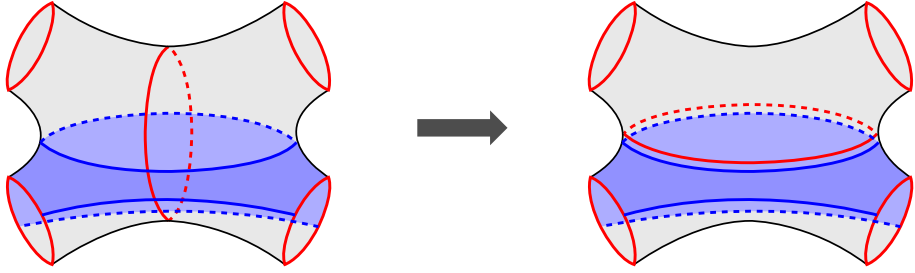}
    \caption{Pants move from $\nu_{r_1-1}$ to $\nu_{r_1}$ with support $S_{r_1}$.  The curve $\gamma_{r_1}'$ is isotopic to the curve component $q$ of $\pd Q_1$.}
    \label{first}
\end{figure}

Observe that $\n \subset (\pd X_{r_1} \cap Y_1) \subset (\nu_0 \cap \nu_{r_1}) \cap Y_1$, so by Lemma \ref{diam}, $d_{Y_1}(\nu_0,\beta) \leq 2$.  If $[\n] \in [\nu_m \cap Y_1]$, then there is an arc $\delta \subset \nu_m \cap R$ isotopic to $\beta$, and $d_{Y_1}(\nu_0,\delta) \leq 2$, completing the proof.  Therefore, we suppose that $[\beta] \notin [\nu_m \cap Y_1]$.  Let $l$ denote the smallest index such that $[\beta] \notin [\nu_l \cap Y_1]$, let $X_{l-1}'$ be the component of $X_{l-1}$ containing $Q_1$, and let $X_l'$ be the component of $X_l$ containing $X_{l-1}'$.  Since $\beta \subset \nu_{r_1} \cap Y_1$, we have $r_1 < l$, and by previous assumptions, $\cci_{X_l'}(Y_1) = 1$ or $2$ (since its value is an integer).  We have assumed $[\beta] \in [\nu_{l-1} \cap Y_1]$ but $[\beta] \notin [\nu_l \cap Y_1]$, which implies that $[\beta] \in [\gamma_l \cap Y_1]$.  By Lemma \ref{diam}, we have that for any arc $\delta'$ of $\nu_0 \cap Y_1$, either $d_{Y_1}(\delta',\A) \leq 1$ or $d_{Y_1}(\delta',\beta) \leq 1$.  It follows that $d_{Y_1}(\nu_0,\nu_{l-1}) = 1$.  If, in addition, $d_{Y_1}(\nu_0,\nu_l) = 1$, then $\cci_{X_l'}(Y_1) = 1$ and $X_l' \cap Y_1$ contains a rectangular annulus $D$ such that $Q_1 \subset D$ and thus $[\beta] \in [\pd X_l' \cap D] \subset [\nu_l \cap Y_1]$, a contradiction. \\

It follows that $d_{Y_1}(\nu_0,\nu_l) = 2$, so that $X_l' \cap Y_1$ contains a rectangular pair of pants, which is necessarily contained in $R$.  Also, $X_l' \setminus \eta(\gamma_l)$ has two components, $X_{l-1}'$ and another component we will call $S$, where $\cci_S(Y_1) = 1$ and $S \cap Y_1$ contains a rectangular annulus $Q'$.  Let $q'$ denote the boundary curve of $Q'$, noting that $q'$ is isotopic to $q_2$. \\

 If $S$ is a pair of pants, $q'$ is isotopic to a curve in $\pd X_l$.  As $Q' \subset S_l$ we have that $\gamma_l'$ meets $Q'$ by Lemma \ref{intersections}, and since $q_1$ is isotopic to $\gamma_r'$, both $\gamma_l$ and $\gamma_l'$ avoid $q_1$ and $q_2$.  Further, $\nu_{l-1} \cap X_m$ and $\nu_l \cap X_m$ are pants decompositions of $X_m$ which avoid $q_1$ and $q_2$, so we satisfy the hypotheses of Lemma \ref{seam2}. Following the proof of Lemma \ref{seam2}, we have that $\nu_{l-1} \cap R$ contains seams isotopic to $\A$ and $\beta$.  If $\delta$ is an arc of $\gamma_l' \cap Q_1 \subset \gamma_l' \cap R$, then $\delta$ is disjoint from $\A$ and either $\delta$ disjoint from $\beta$, or $\delta$ is one of $\delta_1$, $\delta_2$, or $\delta_3$, where $\delta_1$ is a seam meeting $\beta$ once, $\delta_2$ is a wave meeting $\beta$ once, and $\delta_3$ is a wave meeting $\beta$ twice, as pictured in Figure \ref{cand}.  Lemma \ref{seam2} implies that for such $\delta$, we have $d_{Y_1}(\A,\delta) \leq 1$ unless $\delta = \delta_1$, in which case $d_{Y_1}(\A,\delta_1) = 2$, and for all $\delta$, $d_{Y_1}(\beta,\delta) \leq 2$.  Now, let $\delta'$ be any arc of $\nu_0 \cap Y_1$ such that $[\delta'] \neq [\A],[\beta]$.  As $[\A],[\beta] \in [\nu_0 \cap Y_1]$, we may regard $\delta'$ as an arc in $R$ disjoint from $\A$ and $\beta$, so by Lemma \ref{diam}, $d_{Y_1}(\delta',\A) = 1$.  For $\delta \neq \delta_1$,
\begin{equation}\label{discuss1}
d_{Y_1}(\delta',\delta) \leq d_{Y_1}(\delta',\A) + d_{Y_1}(\A,\delta) \leq 2.
\end{equation}
One the other hand, there is a wave $\eps$ disjoint from both $\delta'$ and $\delta_1$, so
\begin{equation}\label{discuss2}
d_{Y_1}(\delta',\delta) \leq d_{Y_1}(\delta',\eps) + d_{Y_1}(\eps,\delta) \leq 2.
\end{equation}
Therefore, $d_{Y_1}(\nu_0,\delta) \leq 2$, completing the proof. \\

Suppose now that $S$ is not a pair of pants, so that $S$ is a component of $X_{l-1}$ that is not $X_{l-1}'$.  As such, we may (temporarily) commute edges in $\Pa(\Sigma)$ so that $S = X_j$ for some $j$.  Recall that $\cci_S(Y_1) = 1$, and since $S = X_j$ is connected, we have that $d_{Y_1}(\nu_0,\nu_j) = 1$ and there exists an index $s$ such that $d_{Y_1}(\nu_0,\nu_{s}) = 1$ but $d_{Y_1}(\nu_0,\nu_{s-1}) = 0$.  A parallel argument to the one above pertaining to $Q_1$ shows that there is a component $Q_2 \subset Q'$ of $X_s \cap Y_1$ contained in the 4-punctured sphere $S_{s}$ such that $Q_2$ is a rectangular annulus whose curve boundary $q_2$ is isotopic to the curve $\gamma_{s}'$, the unique curve in $\nu_s \setminus \nu_{s-1}$.  Now, note that reversing the commutation may result in a reindexing, but the sets of curves $\{\gamma_1',\dots,\gamma_m'\}$ and supports $\{S_1,\dots,S_m\}$ are not changed.  Thus, we conclude that $s$ corresponds to some index $r_2$ after reversing the commutation such that $Q_2 \subset X_{r_2} \cap Y_1$, $Q_2 \subset S_{r_2}$, and the boundary curve $q_2$ of $Q_2$ is isotopic to $\gamma_{r_2}'$. \\

To summarize, we have that $X_m \cap Y_1 = R$, a rectangular annulus, and $R$ is a union of $Q_1$, $Q_2$, and (possibly) some rectangles.  Further, for $i=1,2$ the arc components of $\text{Fr}_{Y_1} Q_i$ are isotopic to $\A$ and $\beta$, and $\gamma_{r_i}'$ is isotopic to $q_i$, the curve boundary component of $Q_i$, which coincides with a curve boundary component of $R$.  See Figure \ref{RRR} for a schematic of $R$.  Since $Q_i \subset S_{r_i}$, it follows that $\Sigma \setminus \eta(\nu_{r_i})$ has a pair of pants component $P_i$ containing $Q_i$.  Since $[\beta] \in [\nu_{r_1} \cap Y_1]$ but $[\beta] \notin [\nu_m \cap Y_1]$, there is (at least) one boundary component $\sigma_i$ of $P_i$ such that $\sigma_i \notin \nu_m$.  If two boundary components $\sigma_i$ and $\sigma_i'$ of $P_i$ are not in $\nu_m$, choose $\sigma_i$ so that there is an index $t_i$ such that $\sigma_i \notin \nu_{t_i}$ but $\sigma_i' \in \nu_{t_i}$. \\

\begin{figure}
  \centering
    \includegraphics[width=.5 \textwidth]{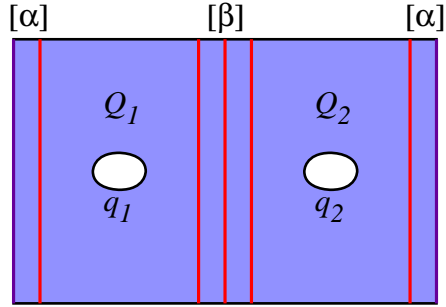}
    \caption{The rectangular annulus $R$.}
    \label{RRR}
\end{figure}

There are two final possibilities to consider which will complete the proof of the theorem.  First, suppose that $\sigma_1 = \sigma_2$ in $\Sigma$ (see Figure \ref{equals}).    In this case, $Q_1 \cup Q_2$ is contained in the support $S_t$ of an elementary move; as such the intersection of $S_t$ and $Y_1$ contains a rectangular pair of pants, which necessarily meets $\gamma_t'$ in an arc $\delta$.  We note that $\nu_{t-1} \cap X_m$ and $\nu_t \cap X_m$ are pants decompositions of $X_m$ meeting the hypotheses of Lemma \ref{seam2}.  Further, as discussed above, since $[\A],[\beta] \in \nu_0 \cap Y_1$, we have that for any arc $\delta'$ of $\nu_0 \cap Y_1$ such that $[\delta'] \neq [\A],[\beta]$, either (\ref{discuss1}) or (\ref{discuss2}) is satisfied.  It follows that $d_{Y_1}(\nu_0,\delta) \leq 2$, completing the proof. \\

\begin{figure}
  \centering
    \includegraphics[width=.65 \textwidth]{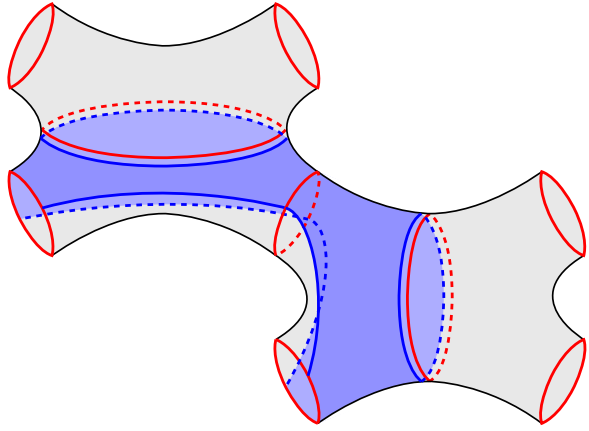}
    \caption{A possible embedding of $R$ in $X_m$ in the case that $\sigma_1 = \sigma_2$.}
    \label{equals}
\end{figure}

The other possibility is that $\sigma_1 \neq \sigma_2$ (see Figure \ref{nequals}).  In this case, we may suppose without loss of generality that $\sigma_1$ is replaced in an elementary move before the replacement of $\sigma_2$, so that there is an index $t$ such that $\sigma_1 = \gamma_t \notin \nu_t$ and $\sigma_2 \in \nu_t$.  Thus, $\gamma_t$ is a boundary component of $P_1$ and we have $Q_1 \subset P_1 \subset S_t$.  Further, $q_1$ is isotopic to a curve in $\pd S_t$, so by Lemma \ref{intersections}, $\gamma_t'$ meets $Q_1$ and there is an arc $\delta \subset \gamma_t' \cap (Y_1 \setminus \eta(Q_2))$.  Let $\delta'$ be any arc of $\nu_0 \cap Y_1$.  As $[\A],[\beta] \in [\nu_0 \cap Y_1]$, we have that either $\delta' \subset Q_2$ or $\delta \cap Q_2 = \emp$.  In the first case, $d_{Y_1}(\delta',\delta) \leq 2$ by Lemma \ref{diam}.  Otherwise, $\delta \subset Y_1 \setminus \eta(Q_2)$, and since $Y_1 \setminus \eta(Q_2)$ is a rectangular annulus, we may invoke Lemma \ref{seam1} to conclude that $d_{Y_1}(\delta',\delta) \leq 2$.  In any case, we have $d_{Y_1}(\nu_0,\delta) \leq 2$, as desired.
\begin{figure}
  \centering
    \includegraphics[width=1 \textwidth]{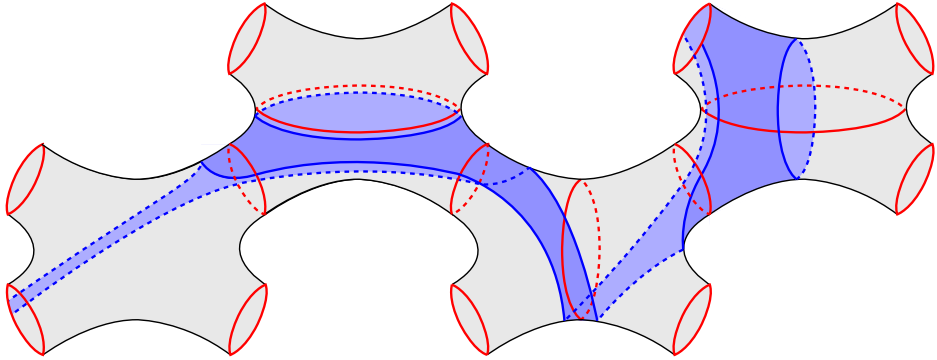}
    \caption{A possible embedding of $R$ in $X_m$ in the case that $\sigma_1 \neq \sigma_2$.  Here we have shown the case in which $\sigma_1 = \gamma_t$ is replaced by an elementary move before $\gamma_{r_2}$ (far right) is replaced by $\gamma_{r_2}'$, isotopic to $q_2$.}
    \label{nequals}
\end{figure}
\end{proof}
\end{claim}}
This exhausts all possible cases, completing the proof of the lemma.
\end{proof}
\end{lemma}

Before proceeding, we note that Lemma \ref{support-contain} holds for paths of any length, not only those of length $\cci(Y)$ or greater.

\begin{proof}[Proof of the Main Theorem]
Let $\nu_0,\ldots, \nu_p$ be a geodesic in $\Pa(\Sigma)$ such that $\nu_0$ and $\nu_p$ contain $Q$, with $Y = \bigcup Y_i$ the complementary subsurface of $Q$.  Recall that $d_Q(\nu_0,\nu_p) = d_Y(\nu_0,\nu_p)$, and suppose towards a contradiction that $\nu_1$ does not contain $Q$.  Then any curve $\gamma$ in $\nu_0 \cap \text{int}(Y)$ is also in $\nu_1$; hence $\pi_Y(\nu_0) = \pi_Y(\nu_1)$.  Let $q$ be an index such that
\[ d_Y(\nu_1,\nu_q) \leq q-1,\]
chosen so that $q$ is maximal with respect to all possible commutations of the path $\nu_1,\dots,\nu_p$.  We will prove that $q = p$, which implies that 
$$d(\nu_0,\nu_p) \le d_{Q}(\nu_0,\nu_p) = d_Y(\nu_0,\nu_p) = d_Y(\nu_1,\nu_p) = p -1.$$
This will contradict the assumption that $\nu_0,\dots,\nu_p$ is a geodesic and will complete the proof of the theorem. \\

Trivially, $q \geq 1$.  Let $X_j$ denote the support of the path $\nu_q,\dots,\nu_p$, where $q \leq j \leq p$.  By Lemma \ref{support-contain}, if $X_p$ contains a component $Y_i$ of $Y$, then after a possible commutation of edges there exists $q'$ such that $q < q' \leq p$ and $d_Y(\nu_q,\nu_{q'}) \leq q' - q$, and thus
\[ d_Y(\nu_1,\nu_{q'}) \leq d_Y(\nu_1,\nu_q) + d_Y(\nu_q,\nu_{q'}) \leq q' - 1,\]
contradicting the maximality of $q$.  Thus, we may assume that $X_p$ does not contain a component $Y_i$ of $Y$. \\
 
We will show that $\nu_p \cap Y \subset \nu_q \cap Y$.    Suppose by way of contradiction that there is an index $i$ and a curve $\gamma' \subset \nu_p \cap Y_i$ such that $\gamma' \notin \nu_q$.  Then $\gamma' \subset \text{int}(X_p)$.  For each curve $\gamma'' \subset \pd Y_i$, either $\gamma'' \notin \nu_q$, in which case $\gamma'' \subset \text{int}(X_p)$ or $\gamma'' \in \nu_q$, in which case $\gamma'' \subset \pd X_p$.  In any case, we have $Y_i \subset X_p$, a contradiction. \\

Since $\nu_p \cap Y$ contains only curves, $\nu_p \cap Y \subset \nu_q \cap Y$ implies $\pi_Y(\nu_p) = \pi_Y(\nu_q)$ and $d_Y(\nu_q,\nu_p) = 0$; thus by the maximality of $q$, we conclude $q = p$. Hence, as noted above, we have that $d(\nu_0,\nu_p) \le p-1$, contradicting that our path is a geodesic and completing the proof.
\end{proof}

\bibliographystyle{amsalpha}
\bibliography{bib_pants_flats.bbl}

\providecommand{\bysame}{\leavevmode\hbox to3em{\hrulefill}\thinspace}
\providecommand{\MR}{\relax\ifhmode\unskip\space\fi MR }
\providecommand{\MRhref}[2]{%
  \href{http://www.ams.org/mathscinet-getitem?mr=#1}{#2}
}
\providecommand{\href}[2]{#2}
\begin{thebibliography}{ALPS12}

\bibitem[ALPS12]{ALPS}
Javier Aramayona, Cyril Lecuire, Hugo Parlier, and Kenneth~J Shackleton,
  \emph{Convexity of strata in diagonal pants graphs of surfaces}, Publicacions
  Matem{\`a}tiques \textbf{57} (2012), no.~1, 219--237.

\bibitem[APS08]{APS1}
Javier Aramayona, Hugo Parlier, and Kenneth~J Shackleton, \emph{Totally
  geodesic subgraphs of the pants complex}, Math. Res. Lett \textbf{15} (2008),
  no.~2, 309--320.

\bibitem[APS09]{APS2}
Javier Aramayona, Hugo Parlier, and Kenneth Shackleton, \emph{Constructing
  convex planes in the pants complex}, Proceedings of the American Mathematical
  Society \textbf{137} (2009), no.~10, 3523--3531.

\bibitem[Ara10]{Ara}
Javier Aramayona, \emph{Simplicial embeddings between pants graphs}, Geometriae
  Dedicata \textbf{144} (2010), no.~1, 115--128.

\bibitem[BF06]{BF}
Jeffrey Brock and Benson Farb, \emph{Curvature and rank of {T}eichm{\"u}ller
  space}, American journal of mathematics (2006), 1--22.

\bibitem[BM07]{BMar}
Jeffrey Brock and Dan Margalit, \emph{Weil--petersson isometries via the pants
  complex}, Proceedings of the American Mathematical Society \textbf{135}
  (2007), no.~3, 795--803.

\bibitem[BM08]{BM}
Jason~A Behrstock and Yair~N Minsky, \emph{Dimension and rank for mapping class
  groups}, Annals of Mathematics \textbf{167} (2008), 1055--1077.

\bibitem[BMM11]{BMMII}
Jeffrey Brock, Howard Masur, and Yair Minsky, \emph{Asymptotics of
  weil--petersson geodesics ii: bounded geometry and unbounded entropy},
  Geometric and Functional Analysis \textbf{21} (2011), no.~4, 820--850.

\bibitem[Bro01]{Brock2}
Jeffrey Brock, \emph{Weil-{P}etersson translation distance and volumes of
  mapping tori}, Comm. Anal. Geom (2001).

\bibitem[Bro03]{Brock1}
\bysame, \emph{The {W}eil-{P}etersson metric and volumes of 3-dimensional
  hyperbolic convex cores}, Journal of the American Mathematical Society
  \textbf{16} (2003), no.~3, 495--535.

\bibitem[Est13]{Est}
Jos{\'e}~L Est{\'e}vez, \emph{Large flats in the pants graph}, arXiv preprint
  arXiv:1306.3170 (2013).

\bibitem[Ham05]{Ham}
Ursula Hamenstaedt, \emph{Geometry of the mapping class groups {III}:
  {Q}uasi-isometric rigidity}, arXiv preprint math/0512429 (2005).

\bibitem[Joh06]{JJheg}
Jesse Johnson, \emph{Heegaard splittings and the pants complex}, Algebraic \&
  Geometric Topology \textbf{6} (2006), 853--874.

\bibitem[Mar04]{Mar}
Dan Margalit, \emph{Automorphisms of the pants complex}, Duke Mathematical
  Journal \textbf{121} (2004), no.~3, 457--480.

\bibitem[Mas76]{Mas}
Howard Masur, \emph{The extension of the {W}eil-{P}etersson metric to the
  boundary of {T}eichmuller space}, Duke Mathematical Journal \textbf{43}
  (1976), no.~3, 623--635.

\bibitem[MM00]{MMII}
Howard~A Masur and Yair~N Minsky, \emph{Geometry of the complex of curves {II}:
  {H}ierarchical structure}, Geometric and Functional Analysis \textbf{10}
  (2000), no.~4, 902--974.

\bibitem[Wol03]{Wol}
Scott~A. Wolpert, \emph{Geometry of the {W}eil-{P}etersson completion of
  {T}eichm\"{u}ller space}, Surveys in Differential Geometry \textbf{VIII}
  (2003), 357--393.

\bibitem[Zup13]{Zuppy}
Alexander Zupan, \emph{Bridge and pants complexities of knots}, Journal of the
  London Mathematical Society \textbf{87} (2013), no.~1, 43--68.

\end{thebibliography}

\end{document}